\documentclass{article}
\usepackage{color}
\usepackage{amsmath,amssymb,amsthm}
\usepackage{hyperref}
\usepackage{aliascnt}
\newtheorem{theorem}{Theorem}[section]

\newaliascnt{lemma}{theorem}
\newtheorem{lemma}[lemma]{Lemma}
\aliascntresetthe{lemma}

\newaliascnt{proposition}{theorem}
\newtheorem{proposition}[proposition]{Proposition}
\aliascntresetthe{proposition}

\newaliascnt{corollary}{theorem}
\newtheorem{corollary}[corollary]{Corollary}
\aliascntresetthe{corollary}

\newaliascnt{definition}{theorem}
\newtheorem{definition}[definition]{Definition}
\aliascntresetthe{definition}

\newaliascnt{example}{theorem}
\newtheorem{example}[example]{Example}
\aliascntresetthe{example}

\newaliascnt{remark}{theorem}
\newtheorem{remark}[remark]{Remark}
\aliascntresetthe{remark}

\usepackage{arxiv}

\usepackage[utf8]{inputenc} % allow utf-8 input
\usepackage[T1]{fontenc}    % use 8-bit T1 fonts
\usepackage{hyperref}       % hyperlinks
\usepackage{url}            % simple URL typesetting
\usepackage{booktabs}       % professional-quality tables
\usepackage{nicefrac}       % compact symbols for 1/2, etc.
\usepackage{microtype}      % microtypography
\usepackage{lipsum}
\usepackage{graphicx}
\graphicspath{ {./images/} }
\usepackage{graphicx}
\usepackage{epstopdf}
\usepackage{algorithmic}

\title{From Nonsmooth Minima to Smooth Branches via Heat Kernel Regularization}

\author{
 Hyeontae Jo \\
  Department of Mathematics, Ajou University, Yeongtong-gu, Suwon, 16499, Republic of Korea \\
  Biomedical Mathematics Group, Pioneer Research Center for Mathematical and Computational Sciences,\\
  Institute for Basic Science, Yuseong-gu, Daejeon, 34126, Republic of Korea\\
  \texttt{ajouhtj@ajou.ac.kr} \\
}

\begin{document}
\maketitle
\begin{abstract}
Many optimization problems in science and engineering involve objective functions that are nonsmooth at their minimizers. A common strategy is to trace a branch of minimizers of a regularized objective as the smoothing scale tends to zero; however, for nonsmooth functions, it is generally unclear whether such a branch can be continued and whether the associated continuation equation remains locally solvable. We study heat-kernel regularization and the resulting continuation equation along a local minimizing branch connected to a minimizer of the original objective. Under a global growth condition and a local leading-order description of the form $|x|^a$ with $1 \le a \le 2$, we first show that the regularized objective admits global minimizers and that any such minimizing branch localizes at the natural heat scale $O(\sqrt{t})$. We then prove that the asymptotic behavior of the regularized Hessian is determined by the local profile of the original objective: it remains uniformly positive definite in the quadratic case $a=2$, while in the subquadratic regime $1 \le a < 2$ its smallest eigenvalue grows at the controlled rate $t^{(a-2)/2}$. Consequently, the regularized Hessian remains asymptotically nondegenerate for all sufficiently small $t>0$, and the continuation equation remains locally solvable, even when the original objective does not admit a classical Hessian at the minimizer. Our results provide a rigorous second-order framework for continuation-based analysis in nonsmooth optimization by showing how heat regularization restores nondegeneracy near singular minimizers.
\end{abstract}

\section{Introduction}
Finding the minimizers of an objective function $f:\mathbb{R}^n\to\mathbb{R}^{\ge0}$ is a cornerstone issue not only in optimization theory but also in modern deep learning. In many applications, the objective $f$ is nonsmooth or exhibits singular geometries near its minimizers~\cite{fuduli2004minimizing,nesterov2005smooth,kiwiel2007convergence,ma2022beyond,jia2025first, jo2026neural}. This lack of classical differentiability renders standard gradient-based analyses and second-order methods inapplicable, posing significant challenges for both theory and practice~\cite{garmanjani2013smoothing, wang2022differentially}. A natural remedy is to approximate $f$ using heat-kernel regularization. For $t>0$ we define the smoothed objective as the convolution:
\begin{equation}\label{eq:def-Pt}
(P_t f)(x):=\int_{\mathbb{R}^n} G_t(x-y)\,f(y)\,dy,
\end{equation}
where $G_t(y):=\frac{1}{(4\pi t)^{n/2}}e^{-\frac{\|y\|_2^2}{4t}}$ denotes the heat kernel and $\|y\|_2:=\Bigl(\sum_{i=1}^n y_i^2\Bigr)^{1/2}$ denotes the Euclidean norm of $y=(y_1,\dots,y_n)\in\mathbb R^n$. Then the family \(\{P_t f\}_{t>0}\) provides a smooth approximation of \(f\), and the regularization parameter $t$ acts as a homotopy variable, connecting a highly regularized landscape ($t>0$) to the original nonsmooth objective, in the sense that $\lim_{t\to0^+} P_t f(x)=f(x)$ at every continuity point of $f$ (See Chapter 3 in \cite{stein2009real} or Chapter 9 in \cite{brezis2011functional}). 

Variants of this smoothing technique have been broadly applied across diverse fields, including imaging, geometric data analysis, inverse problems, and kernel-based optimization~\cite{seo2010heat, lindeberg2013scale, yang2014mollification, nesterov2017random, chung2019heat}. In practice, these applications frequently rely on a single-scale approximation—fixing a specific, sufficiently small $t > 0$—to simplify the optimization landscape. However, such a static approach inherently introduces a smoothing bias, as the minimizer of $P_t f$ does not necessarily coincide with that of the original objective $f$. To eliminate this bias and achieve a more refined analysis, one can instead track a branch of critical points $x_t$ satisfying $\nabla_x(P_t f)(x_t) = 0$ as $t$ converges to $0^+$~\cite{iwakiri2022single, starnes2023gaussian, xu2025global}. Specifically, the implicit function theorem yields a local critical-point branch \cite{allgower1993continuation,hao2022adaptive,seguin2022continuation}. Differentiating the identity $\nabla_x(P_t f)(x_t)=0$ with respect to the regularization parameter $t$, one formally arrives at the Gaussian continuation equation~\cite{mobahi2015theoretical,ilersich2025deep,gu2026deep}.
\begin{equation}\label{eq:continuation-raw}
\partial_t \nabla_x(P_t f)(x_t)+\nabla_x^2(P_t f)(x_t)\,\dot x_t=0.
\end{equation}
Here, $\nabla_x$ and $\nabla_x^2$ denote the gradient and Hessian with respect to the spatial variable $x$, respectively. Whenever the Hessian $\nabla_x^2(P_t f)(x_t)$ is nonsingular, the evolution of the minimizing branch can be explicitly solved as
\begin{equation}\label{eq:continuation-inverse}
\dot x_t = -\bigl(\nabla_x^2(P_t f)(x_t)\bigr)^{-1}\partial_t \nabla_x(P_t f)(x_t).
\end{equation}
This identity makes clear that the continuation of the minimizer branch is strictly governed by the inverse Hessian of $P_t f$. Consequently, the path $t \mapsto x_t$ is locally well-defined and numerically stable when the regularized Hessian remains nondegenerate. However, for nonsmooth objectives, the invertibility of $\nabla_x^2(P_t f)$ is not a priori guaranteed; as $t \to 0^+$, the smoothing effect of the heat kernel vanishes, potentially leading to a singular or ill-conditioned Hessian at the limit.

The primary objective of this paper is to rigorously characterize how heat-kernel regularization retains nondegeneracy along the minimizing branch. We aim to show that even if the original objective $f$ lacks a classical Hessian at the minimizer, the convolution $P_t f$ generates a "surrogate" curvature whose asymptotic behavior is precisely determined by the local leading-order profile of $f$. By establishing this, we prove that the continuation equation remains locally solvable for all sufficiently small $t > 0$, providing a robust theoretical foundation for tracing minimizing branches from nonsmooth minima.

Previous studies have shown that global convexity is preserved under heat-kernel regularization: if $f$ is globally convex, then so is $P_t f$~\cite{ishige2026preservation}. While such results provide strong global structural guarantees, they often rely on idealized assumptions that are rarely satisfied in complex, non-convex optimization landscapes~\cite{ma2022beyond, jo2026neural}. Existing analyses remain limited in their ability to describe the behavior of the regularized Hessian near strictly nonsmooth or non-convex points. 

To bridge this gap, our work shifts the focus from global properties to a fine-grained local analysis, demonstrating that the restoration of nondegeneracy is a local phenomenon driven by the leading-order profile of the minimizer. This perspective allows us to establish the local solvability of the continuation equation for a broad class of nonsmooth functions, including those with subquadratic cusps where classical second-order information is entirely absent. Specifically, we analyze the objective $f$ under three mild assumptions: (Definition~\ref{ass:global-lower-barrier}) a global quadratic growth condition to ensure coercivity and the unique global minimizer of $f$ \cite{bonnans2013perturbation, drusvyatskiy2013tilt, chieu2021quadratic, davis2025local, khanh2025local}, (Definition~\ref{ass:leading-profile}) a local leading-order profile that characterizes the singularity near the minimizer~\cite{burke1993weak}, and (Definition~\ref{ass:tail-control}) polynomial tail control to guarantee that the convolution is well-defined (see also tempered distributions~\cite{stein1970singular}). 

Our primary objective is to investigate whether heat-kernel regularization induces a nondegenerate Hessian along the minimizing branch:
$$x_t\in \operatorname*{arg\,min}_{x\in\mathbb R^n} P_t f.$$ Under the assumptions above, our analysis first focuses on the small-$t$ regime. In our result, every global minimizer of $P_t f$ is forced to localize near the origin as $t\to0^+$; more precisely, we prove that $\|x_t\|_2 = O(\sqrt{t})$, so that the convergence follows the natural heat scale (Theorem~\ref{thm:existence-localization}). 

Next, our analysis of the nondegeneracy of the regularized Hessian $\nabla_x^2(P_t f)(x_t)$ reveals two distinct regimes determined by the local exponent $a \in [1, 2]$ (Theorem~\ref{thm:main-nondegeneracy}):

\noindent
{\bf 1. The quadratic regime ($a=2$).}
In this case, the minimizer has genuinely quadratic structure near the origin. The regularized Hessian remains at a finite positive scale along the minimizing branch (Example~\ref{ex:quadratic}). More precisely, for all sufficiently small $t>0$,
$$\nabla_x^2(P_t f)(x_t) \succ c_0 I,$$
for some constant $c_0>0$. Here, the symbol $\succ$ indicates that the Hessian is bounded below by a positive definite matrix, ensuring that the regularized objective maintains a strictly convex-like curvature at the minimizer. That is, the Hessian does not vanish as $t\to0^+$. In particular, the result applies even when the original function $f$ fails to have a classical Hessian at the distinguished minimizer (e.g., Example~\ref{ex:nonsmooth-quadratic}). Thus the regularized landscape retains a standard nondegenerate quadratic geometry near the branch.

\noindent {\bf 2. The subquadratic cusp regime ($1\le a<2$).}
In this case, the mechanism of nondegeneracy is fundamentally different from the quadratic case. Instead of remaining at a finite positive scale, the regularized Hessian blows up as $t\to0^+$(Example~\ref{ex:linear},~Example~\ref{ex:nonsmooth-linear}). More precisely, its size is governed by the exponent $a$, with the model scale
$$\nabla_x^2(P_t f)(x_t) \succsim t^{(a-2)/2}I.$$
The symbol $\succsim$ denotes that the minimum eigenvalue of the Hessian grows at least at the rate of $t^{(a-2)/2}$. While the landscape becomes singular, this uniform blow-up ensures that the Hessian remains invertible for any $t>0$, effectively "sharpening" the nondegeneracy. 

Therefore, the local leading-order exponent completely distinguishes the two types of nondegeneracy produced by heat-kernel regularization. The case $a=2$ corresponds to a finite positive-definite Hessian scale, whereas every subquadratic exponent $1\le a<2$ leads to a positive-definite blow-up regime. In both cases, the continuation equation remains meaningful along the minimizing branch for all sufficiently small positive times, although the asymptotic mechanism is different.

In general, however, globally tracking such a branch over the entire range of $t>0$ is challenging in nonconvex settings, since competing local-minimum branches may coexist and exchange global optimality at intermediate smoothing scales (see Example~\ref{ex:discontinuous_branch}, Figure~\ref{fig:3}). Thus, a single globally smooth minimizing continuation path need not exist in full generality. Our analysis therefore focuses on the small-$t$ regime. 

Within this regime, we also identify two further structural consequences. First, for every sufficiently small scale, each global minimizer lies on a unique local $C^1$ critical branch, so that the continuation equation is rigorously valid near the minimizing branch (Proposition~\ref{prop:local-branch}). Second, along any nondegenerate critical branch, the associated branch energy satisfies the identity
$$
\frac{d}{dt}(P_t f)(x_t)=\partial_t(P_t f)(x_t)=\Delta(P_t f)(x_t).
$$
This provides a natural analytic quantity for comparing competing branches (Lemma~\ref{lem:branch-energy}). Furthermore, under a finite-limit hypothesis, continuation past a finite terminal scale is equivalent to the persistence of uniform Hessian nondegeneracy (Theorem~\ref{thm:continuation-alternative}. These results clarify that the principal obstruction to local continuation is not the nonsmoothness of the original objective itself, but the possible breakdown of nondegeneracy or the global competition between distinct branches at larger smoothing scales.

In summary, these results demonstrate that heat-kernel convolution provides a unified framework for transitioning from nonsmooth minima to smooth, traceable branches while preserving the necessary nondegenerate structure. Our findings suggest that continuation-equation-based methods offer a robust analytical foundation for nonsmooth optimization across diverse applications, including solving inverse problems \cite{yang2014mollification}, filtering geometric noise \cite{seo2010heat, lindeberg2013scale, chung2019heat}, resolving ill-conditioned systems \cite{luo2024regularization}, and navigating highly non-convex landscapes in deep learning \cite{xu2025global, gu2026deep}. This provides a systematic starting point for studying singular problems through the lens of regularized second-order geometry.

\section{Definitions} 
For simplicity and without loss of generality, we assume the distinguished minimizer is located at the origin with $f(0) = 0$
\begin{definition}\label{ass:global-lower-barrier} We say that $f:\mathbb R^n\to\mathbb R$ satisfies a global quadratic lower bound if there exists a constant $c>0$ such that
\begin{equation}\label{eq:quadratic_growth}
    f(x)\ge c\|x\|_2^2, \qquad\forall x\in\mathbb R^n.
\end{equation}
\end{definition}
Under this condition, $x=0$ is the unique global minimizer of $f$. For the sake of simplicity in the subsequent analysis, we normalize $f$ by $c$ and assume $c=1$.

As a local assumption, we describe the behavior of $f$ near the minimizer through a leading-order profile together with a small remainder $r$. For $a\in[1,2]$, we introduce the isotropic profile
$$\Phi_a(x):=\sum_{i=1}^n |x_i|^a.$$

\begin{definition}
\label{ass:leading-profile}
We say that $f$ has a local leading-order profile $\Phi_a$ at the origin if there exist $a\in[1,2]$, $r_0>0$, a function $r:\mathbb R^n\to\mathbb R$, and a function
$$\omega:[0,r_0)\to[0,\infty)$$
such that $\lim_{\rho\to0^+}\omega(\rho)=0$ and
$$f(x)=\Phi_a(x)+r(x),\qquad \forall x\in B_{r_0}(0)=\{x\,|\,\|x\|<r_0\},$$
where
$$|r(x)|\le \omega(\|x\|_2)\,\Phi_a(x),\qquad \forall x\in B_{r_0}(0).$$
\end{definition}
\begin{remark}
This assumption implies that for all $x$ sufficiently close to the origin, the objective is bounded as $\frac{1}{2}\Phi_a(x) \le f(x) \le \frac{3}{2}\Phi_a(x)$, meaning $f$ and $\Phi_a$ share the same local scaling. We also exclude $a > 2$ as it would violate the global quadratic growth Eq.\eqref{eq:quadratic_growth} in Definition~\ref{ass:global-lower-barrier} and lead to a vanishing Hessian determinant (e.g., Example~\ref{ex:quadratic}), rendering the continuation equation ill-posed.
\end{remark}
%\begin{uremark}
%The shirinking condition $\lim_{\rho\to0^+}\omega(\rho)=0$ is important, because ...
%\end{uremark}
In addition, we assume that $f$ has at most polynomial growth at infinity, namely, there exist constants $C>0$ and $m\ge2$ such that
\begin{definition}\label{ass:tail-control} 
We say that $f$ has polynomial tail control if there exist constants $C>0$ and $m\ge2$ such that $f(x)\le C\bigl(1+\|x\|_2^m\bigr)$, $\forall x\in\mathbb R^n$.
\end{definition}
This assumption guarantees that $P_t f$ is well defined for all $t>0$; in the present argument, it is used only to control the far-field contribution.
\begin{remark}\label{rem:r-extension}
For later use, we fix a global remainder $r:\mathbb R^n\to\mathbb R$ associated with the local decomposition in Definition~\ref{ass:leading-profile}. More precisely, we choose $r$ so that $f(x)=\Phi_a(x)+r(x)$ for all $x\in B_{r_0}(0)$, and extend it to all of $\mathbb R^n$ in such a way that $r$ has polynomial growth at infinity. This is possible because $f$ has polynomial tail control by Definition~\ref{ass:tail-control}, while $\Phi_a(x)\lesssim 1+\|x\|_2^2$ for $a\in[1,2]$. In particular, after fixing such an extension, there exist constants $C>0$ and $m\ge2$ such that $|r(x)|\le C\bigl(1+\|x\|_2^m\bigr),\qquad \forall x\in\mathbb R^n.$
\end{remark}

\begin{definition}\label{def:critical-branch}
Let $I\subset (0,\infty)$ be an interval. A map $x:I\to\mathbb R^n$ is called a \emph{critical-point branch} of $P_t f$ on $I$ if $x\in C^1(I;\mathbb R^n)$ and
$$
\nabla_x(P_t f)(x_t)=0
\qquad\text{for all }t\in I.
$$
\end{definition}

\begin{definition}\label{def:nondegenerate-branch}
A critical-point branch $x:I\to\mathbb R^n$ is called a \emph{nondegenerate local minimizing branch} if
$$
\nabla_x^2(P_t f)(x_t)\succ 0
\qquad\text{for all }t\in I.
$$
\end{definition}

\begin{definition}\label{def:maximal-branch}
Let $x:(0,T_{\max})\to\mathbb R^n$ be a nondegenerate local minimizing branch. We say that $x$ is \emph{maximal} if there do not exist $\varepsilon>0$ and a nondegenerate local minimizing branch
$$
\widetilde x:(0,T_{\max}+\varepsilon)\to\mathbb R^n
$$
such that
$$
\widetilde x_t=x_t
\qquad\text{for all }t\in(0,T_{\max}).
$$
\end{definition}

\begin{figure}[t!]
  \centering  \label{fig:1}\includegraphics[width=\textwidth]{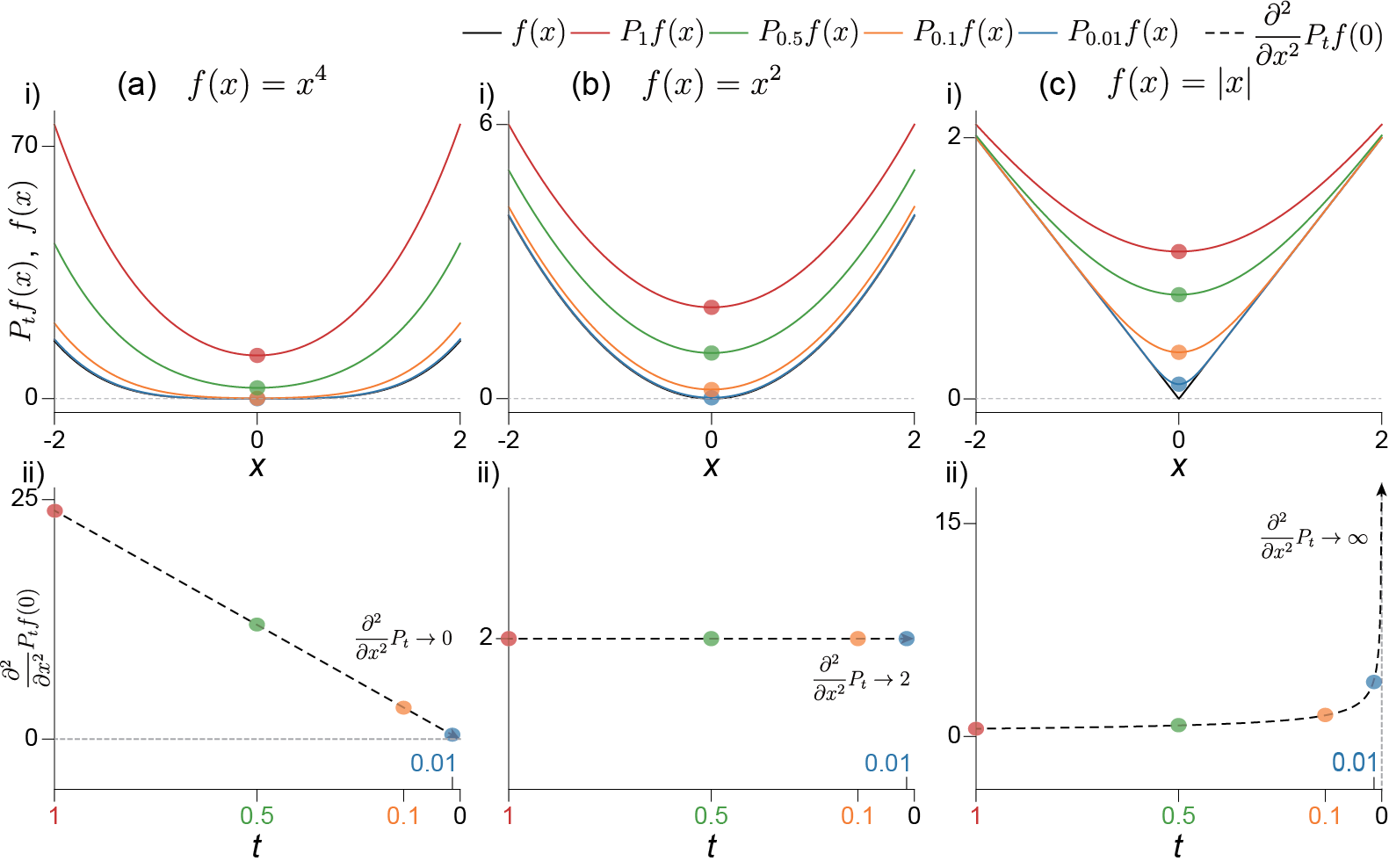}
  \caption{\textbf{Comparison of the second derivatives of $P_t f$ for the prototype functions $f(x)=x^4$, $x^2$, and $|x|$.} (a) i) We plot $f(x)=x^4$ and its heat-kernel regularization $P_t f(x)$ for several values of $t$ (e.g., $t=1,\,0.5,\,0.1,\,0.01$), where the minimizer remains at $x_t=0$. ii) We plot $\frac{\partial^2}{\partial x^2}P_t f(0)$ for $0<t<1$, with colored markers indicating its values at the selected parameter levels. In this case, $\frac{\partial^2}{\partial x^2}P_t f(0)\to 0$ as $t\to0^+$. (b) i) We repeat the same visualization for $f(x)=x^2$, where again the minimizer is $x_t=0$. ii) In contrast to (a), we have $\frac{\partial^2}{\partial x^2}P_t f(0)=2$ for all $0<t<1$, indicating that the Hessian in the corresponding continuation equation is preserved under heat regularization. (c) i) We plot $f(x)=|x|$ and its regularization $P_t f(x)$ at the same parameter levels, with minimizer $x_t=0$. ii) In this case, the second derivative $\frac{\partial^2}{\partial x^2}P_t f(0)$ blows up as $t\to0^+$.}
\end{figure}
\section{Examples} The following simple examples are provided to clearly illustrate the qualitatively different behaviors of the regularized Hessian that arise in our analysis. In particular, they highlight the clear distinction between the quadratic regime, where the Hessian remains at a finite positive scale, and the subquadratic regime, where it blows up as $t\to0^+$. To demonstrate these curvature dynamics transparently, the minimizing branches in the examples are chosen to be trivial ($x_t=0$ for all $t>0$) except the last example~Example~\ref{ex:discontinuous_branch}.

\begin{example}\label{ex:quadratic} For $m\ge1$ and $x\in\mathbb R$, consider $f(x)=x^{2m}.$ Let $u(t,x):=P_t f(x)$. Then $u$ solves the heat equation $\partial_t u=\partial_{xx}u$ with initial condition $u(0,x)=f(x)$. In this case, one has the explicit formula (see also caloric polynomials in \cite{cannon1984one})
$$u(t,x)
=
\sum_{j=0}^{m}\frac{(2m)!}{j!(2m-2j)!}\,x^{2m-2j}\,t^j,$$
so the minimizing branch is given by 
$x_t=0$ for all $t>0$. Differentiating twice with respect to $x$ and evaluating at $0$, we obtain
\begin{equation}\label{eq:Pt-even-power-second}
\partial_{xx}u(t,0)
=
\frac{(2m)!}{(m-1)!}\,t^{m-1}.
\end{equation}
We distinguish two cases.

\noindent
\emph{Case 1: $m>1$ ($f(x)=x^{2m}$).} By Eq.\eqref{eq:Pt-even-power-second},
$$
\frac{\partial^2}{\partial x^2}P_{t}f(0)=\frac{(2m)!}{(m-1)!}\,t^{m-1}\to0
\qquad\text{as }t\to0^+.
$$
Thus the heat-kernel regularization produces a collapse of curvature at the minimizer. In particular, this example shows why the quadratic-growth assumption is essential: it excludes flatter-than-quadratic minima ($a > 2$) for which $\frac{\partial^2}{\partial x^2}P_{t}f(0)$ vanishes, causing the inverse Hessian coefficient in the one-dimensional continuation equation to diverge~(Figure~\ref{fig:1}(a)).

\noindent
\emph{Case 2: $m=1$ ($f(x)=x^2$).} In this case,
$$
\frac{\partial^2}{\partial x^2}P_{t}f(0)=2.
$$
Thus the regularized second derivative remains uniformly positive along the minimizing branch. This is the model example of the regime $a=2$, where the inverse Hessian in the continuation equation~Eq.\eqref{eq:continuation-inverse} stays at a finite nondegenerate scale as $t\to0^+$ (Figure~\ref{fig:1}(b)).
\end{example}

\begin{example}\label{ex:linear} For $x\in\mathbb R$, consider $f(x)=|x|$. Then 
$$ P_t f(x) = \frac{2\sqrt t}{\sqrt\pi}e^{-x^2/(4t)} + |x|\,\operatorname{erf}\!\left(\frac{|x|}{2\sqrt t}\right), $$
and the minimizing branch is given by 
$x_t=0$ for all $t>0$. Moreover,
$$
\frac{\partial^2}{\partial x^2}P_{t}f(x)
=
\frac{1}{\sqrt{\pi t}}e^{-x^2/(4t)},
$$
so that along the minimizing branch ($x_t=0$), $\frac{\partial^2}{\partial x^2}P_{t}f(x_t)=\frac{1}{\sqrt{\pi t}}
\to +\infty$ as $t\to0^+.$ Consequently, the inverse Hessian satisfies $\bigl(\frac{\partial^2}{\partial x^2}P_{t}f(x_t)\bigr)^{-1}=\sqrt{\pi t}\to 0$ as $t\to0^+.$ This blow-up in curvature does not obstruct the analysis; on the contrary, it ensures that the coefficient in the continuation equation remains invertible for all $t>0$. The infinite curvature at the limit effectively locks the branch to the origin, providing a natural continuation branch even when the original function lacks a classical gradient and Hessian. Strictly speaking, the prototype $f(x)=|x|$ does not satisfy the global quadratic-growth assumption~Definition~\ref{ass:global-lower-barrier}. However, this causes no essential difficulty here: one may modify $f$ outside a sufficiently large compact set so that it has quadratic growth at infinity, without changing the local convolution structure or the asymptotic behavior of $\frac{\partial^2}{\partial x^2}P_t f(0)$ near the origin Figure~\ref{fig:1}(c).
\end{example}

\begin{figure}[ht!]
  \centering  \label{fig:2}\includegraphics[width=\textwidth]{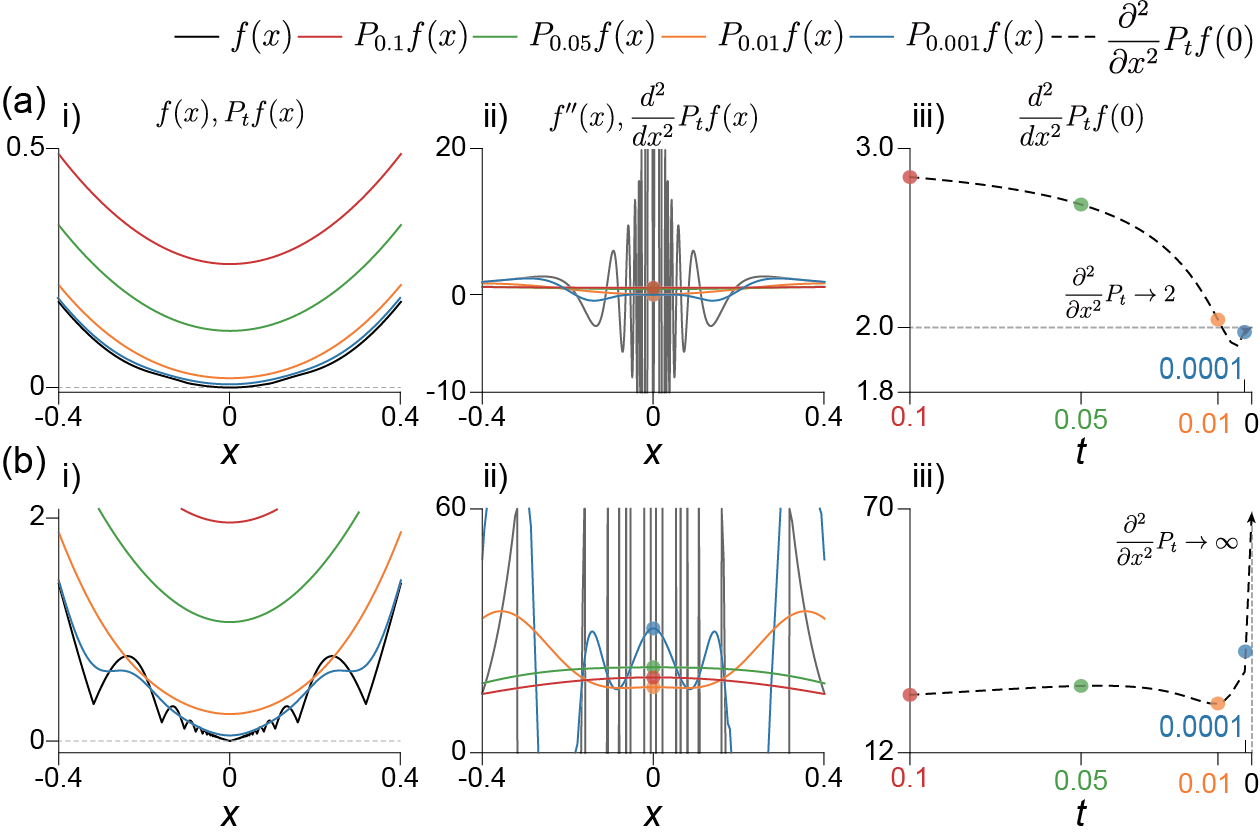}
\caption{\textbf{Nonsmooth examples with local leading-order exponents $a=2$ (Example~\ref{ex:nonsmooth-quadratic}) and $a=1$ (Example~\ref{ex:nonsmooth-linear}), respectively.} 
(a) i) Let $f(0)=0$ and, for $x\neq0$, let $f(x)=x^2+\frac12 x^3\sin(1/x)$. We draw $f(x)$ and its heat-kernel regularization $P_t f$ for $t=0.1,\,0.05,\,0.01,\,0.001$. 
ii) To make the second-derivative behavior more transparent, we also plot the profiles of $f''(x)$ and $\frac{\partial^2}{\partial x^2}P_t f$, where the colored dots indicate the minimizer locations at the selected parameter levels. Notably, $f''(0)$ does not exist. 
iii) We further track the second derivative of $P_t f$ at its minimizer. As a result, $\frac{\partial^2}{\partial x^2}P_t f$ at the minimizer converges to $2$ as $t\to0^+$. 
(b) i) Let $f(0)=0$ and, for $x\neq0$, let $f(x)=|x|\bigl(1+10|x||\sin(1/x)|\bigr)$. We again draw $f(x)$ and its heat-kernel regularization $P_t f$ for the same parameter values. 
ii) We plot the corresponding second-derivative profiles and indicate the minimizer locations by colored dots. 
iii) We then track the second derivative of $P_t f$ at its minimizer. Unlike in (a), $\frac{\partial^2}{\partial x^2}P_t f$ at the minimizer blows up as $t\to0^+$.}
\end{figure}
\begin{example}\label{ex:nonsmooth-quadratic}
Define $f:\mathbb R\to\mathbb R$ by
$$f(x)=
\begin{cases}
0, & x=0,\\
x^2+\frac12 x^3\sin(1/x), & x\ne0.
\end{cases}$$
Then $f$ satisfies Definition~\ref{ass:global-lower-barrier,ass:leading-profile,ass:tail-control} with $a=2$. The first and second derivatives are given by
$$f'(x)=
\begin{cases}
    0, & x=0,\\
2x+\frac12\bigl(3x^2\sin(1/x)-x\cos(1/x)\bigr), & x\ne 0,
\end{cases}
$$
and
$$f''(x)=2+3x\sin(1/x)-2\cos(1/x)-\frac{1}{2x}\sin(1/x).$$
In particular, $f''(0)$ does not exist, since the last term oscillates without limit as $x\to0$.

Nevertheless, for every $t>0$, the heat regularization $P_t f$ is smooth, so $\frac{\partial^2}{\partial x^2}P_{t}f(0)$ is well defined. More precisely,
$$\frac{\partial^2}{\partial x^2}P_{t}f(0)=2+\frac12\int_{\mathbb R}\left(\frac{y^2}{4t^2}-\frac{1}{2t}\right)G_t(y)\,y^3\sin(1/y)\,dy.$$
After the change of variables $y=\sqrt t\,z$, this becomes
$$\frac{\partial^2}{\partial x^2}P_{t}f(0)=2+\frac{\sqrt t}{2\sqrt{4\pi}}\int_{\mathbb R}\left(\frac{z^2}{4}-\frac12\right)e^{-z^2/4}z^3\sin\left(\frac{1}{\sqrt t\,z}\right)\,dz.$$
Since the integral is uniformly bounded, it follows that
$$\frac{\partial^2}{\partial x^2}P_{t}f(0)\to 2\qquad\text{as }t\to0^+.$$

Thus the original Hessian at the minimizer does not exist, yet the heat-regularized Hessian is invertible for every $t>0$ and converges to a finite positive limit (Figure~\ref{fig:2}(a)).
\end{example}

\begin{example}\label{ex:nonsmooth-linear}
Consider
$$
f(x)=
\begin{cases}
0, & x=0,\\
|x|\bigl(1+10|x||\sin(1/x)|\bigr), & x\ne0.
\end{cases}
$$
So the local leading-order profile is $\Phi_1(x)=|x|$ (the local exponent is $a=1$). Thus this gives a nonsmooth example illustrating the linear regime and the corresponding $t^{-1/2}$ Hessian scale under heat-kernel regularization (Figure~\ref{fig:2}(b)).

Although this function does not satisfy the global quadratic-growth assumption in its present form, it can be modified outside a neighborhood of the origin, for instance by prescribing a quadratic-growth branch for $|x|>1$, without changing the local leading-order behavior near the origin.
\end{example}

\begin{figure}[t!]
  \centering  \label{fig:3}\includegraphics[width=\textwidth]{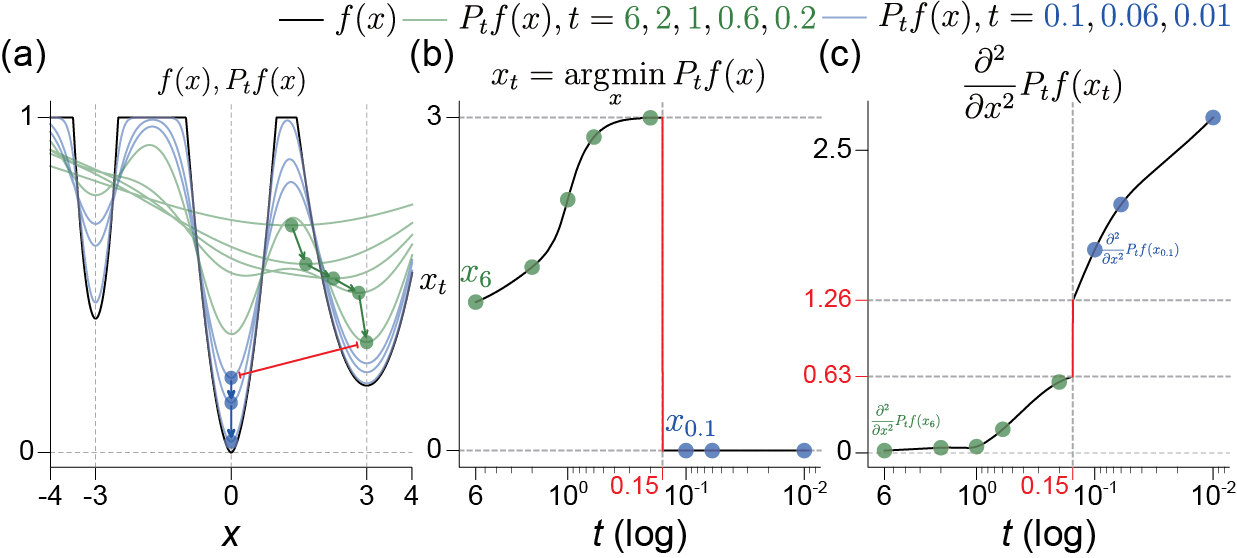}
  \caption{\textbf{Discontinuous minimizing branch $t\mapsto x_t$ of $P_t f(x)$ when $f$ has multiple local minima.} (a) For $f$ in Example~\ref{ex:discontinuous_branch}, which has local minima near $x=-3$, $0$, and $3$, we first plot $f(x)$ together with $P_t f(x)$ for selected parameter values. The global minimizers of $P_t f$ are marked by green dots for $t=6,\,2,\,0.6,\,0.2$ and by blue dots for $t=0.1,\,0.06,\,0.01$. This illustrates that the minimizing branch $t\mapsto x_t$ can undergo a discontinuous jump as $t$ varies. (b) We then plot the minimizer map $t\mapsto x_t$, which reveals a clear branch switch near $t\approx 0.15$ (red line). (c) We also track $\frac{\partial^2}{\partial x^2}P_t f(x_t)$ along the minimizing branch. In a neighborhood of the discontinuity point ($t\approx 0.15$), the second derivative changes abruptly, increasing from approximately $0.63$ to $1.26$. For $t<0.15$, the branch becomes continuous again, and the second derivative blows up as $t\to0^+$. This example shows that the Gaussian continuation equation need not define a globally smooth minimizing path, although for sufficiently small $t>0$ one may still identify a local branch that converges to the global minimizer of the original objective $f$.}
\end{figure}

\begin{example}\label{ex:discontinuous_branch}
Define $f:\mathbb R\to\mathbb R$ by
$$
f(x)=
\begin{cases}
|x|^{1.7}, & |x|\le 1,\\
0.4+0.6(2|x+3|)^2, & |x+3|\le 0.5,\\
0.2+0.8\left(\dfrac{|x-3|}{1.55}\right)^2, & |x-3|\le 1.55,\\
x^2, & |x|\ \text{sufficiently large},\\
1, & \text{otherwise}.
\end{cases}
$$
Then $f(0)=0$, and $f(x)>0$ for every $x\neq 0$, so $x=0$ is the unique global minimizer of $f$. Moreover, near the origin the leading-order profile is $f(x)\sim |x|^{1.7}$, so the local exponent is $a=1.7$. $f$ also contains two asymmetric side valleys centered near $x=-3$ and $x=3$ (Figure~\ref{fig:3}(a)). The left valley is narrower and shallower, while the right valley is wider and deeper. As a consequence, the global minimizer of $P_t f$ need not remain near the origin for intermediate values of $t$. Indeed, $P_t f$ may favor the wider right-hand valley over the narrow minimum at the origin, even though the latter is the unique minimizer of the original objective.

Thus, if one selects $x_t\in \operatorname*{arg\,min}_{x\in\mathbb R} P_t f$, the resulting minimizing path $t\mapsto x_t$ need not form a single globally smooth continuation branch (arrows in Figure~\ref{fig:3}(a)). In particular, as the smoothing scale varies, competing local-minimum branches may exchange global optimality, so that the selected global minimizer can jump discontinuously from one branch to another (Figure~\ref{fig:3}(b)). Correspondingly, quantities evaluated along the selected minimizer, such as $\frac{\partial^2}{\partial x^2}P_{t}f(x_t)$, may also exhibit jump discontinuities as functions of $t$, even though $P_t f$ is smooth in $(t,x)$ for every $t>0$ (Figure~\ref{fig:3}(c)).

This example shows that, in nonconvex settings, globally tracking the minimizing branch of $P_t f$ over all smoothing scales can be genuinely delicate. Nevertheless, in the small-$t$ regime, the heat-regularized minimizers still localize near the original minimizer at the origin, in accordance with the general localization theory.
\end{example}

\section{Main results}
Throughout this section, we assume that $f$ satisfies Definitions~\ref{ass:global-lower-barrier},~\ref{ass:leading-profile}, and \ref{ass:tail-control}.
\subsection{Existence and localization of minimizers}
\begin{lemma}\label{lem:gaussian-lower-bound}
For every $t>0$ and every $x\in\mathbb R^n$, $$(P_t f)(x)\ge \|x\|_2^2+2nt.$$ 
\end{lemma}

\begin{proof}
Since the heat kernel $G_t$ in Eq.\eqref{eq:def-Pt} is nonnegative, applying $P_t$ to both sides of Eq.\eqref{eq:quadratic_growth} yields
$$(P_t f)(x)\ge P_t(\|\cdot\|_2^2)(x).$$ 
Moreover, $P_t(\|\cdot\|_2^2)(x)=\|x\|_2^2+2nt$ (see also Example~\ref{ex:quadratic} for the case $n=1$). Therefore, $(P_t f)(x)\ge \|x\|_2^2+2nt$, which proves the result.
\end{proof}

\begin{corollary}
For each $t>0$, the function $P_t f$ is coercive and therefore admits at least one global minimizer
$$x_t\in \operatorname*{argmin}_{x\in\mathbb R^n} P_t f.$$
Moreover, any such global minimizer satisfies
$$\nabla(P_t f)(x_t)=0.$$
\end{corollary}

\begin{proof}
By Lemma~\ref{lem:gaussian-lower-bound}, $(P_t f)(x)\to\infty$ as $\|x\|_2\to\infty$, so $P_t f$ is coercive. Since $P_t f$ is continuous for every $t>0$, it attains a global minimum at some point. Moreover, $P_t f$ is smooth for every $t>0$, thus any global minimizer $x_t$ satisfies $\nabla(P_t f)(x_t)=0$.
\end{proof}

\begin{lemma}\label{lem:Ptf0-estimate}
As $t\to0^+$, $P_t f(0)=O(t^{a/2})$.
\end{lemma}

\begin{proof}
Decompose $P_t f(0)$ into a local contribution $I_1$ and a far-field contribution $I_2$:
$$P_t f(0)=\int_{B_{r_0}(0)} G_t(y)\,f(y)\,dy+\int_{\mathbb R^n\setminus B_{r_0}(0)} G_t(y)\,f(y)\,dy=:I_1(t)+I_2(t).$$
For the local part $I_1(t)$, using $f(y)\le \frac{3}{2}\Phi_a(y)$ in $B_{r_0}(0)$ (Definition~\ref{ass:leading-profile}), we have
$$I_1(t)\le \frac{3}{2}\int_{B_{r_0}(0)} G_t(y)\,\Phi_a(y)\,dy.$$
Since $\Phi_a(y)=\sum_{i=1}^n |y_i|^a\le C\|y\|_2^a$ for some constant $C>0$, it follows that
$$I_1(t)\le C\int_{\mathbb R^n} G_t(y)\,\|y\|_2^a\,dy.$$
Using the change of variables $y=\sqrt t\,z$, we get
$$\int_{\mathbb R^n} G_t(y)\,\|y\|_2^a\,dy=t^{a/2}\int_{\mathbb R^n}\frac{1}{(4\pi)^{n/2}}e^{-\|z\|_2^2/4}\,\|z\|_2^a\,dz.$$
Since the last integral is finite, we conclude that $|I_1(t)|\le Ct^{a/2}$, and hence $I_1(t)=O(t^{a/2})$.

For the far-field part $I_2(t)$, the polynomial tail control assumption (Definition~\ref{ass:tail-control}) implies that there exist constants $C>0$ and $m\ge 2$ such that
$$f(y)\le C(1+\|y\|_2^m)\qquad \text{for all } y\in\mathbb R^n.$$
Hence
$$I_2(t)\le C\int_{\|y\|_2\ge r_0} G_t(y)\,(1+\|y\|_2^m)\,dy.$$
Using the change of variables $y=\sqrt{t}\,z$, we obtain
$$I_2(t)\le C\int_{\|z\|_2\ge r_0/\sqrt{t}} \frac{1}{(4\pi)^{n/2}}e^{-\|z\|_2^2/4}\,(1+t^{m/2}\|z\|_2^m)\,dz.$$
Since $t\to0^+$, the lower bound $r_0/\sqrt{t}$ tends to infinity. We now estimate the two terms separately. First,
\begin{equation}\label{eq:gaussian_tail}
    \int_{\|z\|_2\ge r_0/\sqrt{t}} e^{-\|z\|_2^2/4}\,dz
\le C e^{-r_0^2/(8t)},
\end{equation}
where we used the standard Gaussian tail bound. For the second term, note that for sufficiently large $\|z\|_2$, $\|z\|_2^m e^{-\|z\|_2^2/4}\le C e^{-\|z\|_2^2/8}$. Therefore,
$$t^{m/2}\int_{\|z\|_2\ge r_0/\sqrt{t}} e^{-\|z\|_2^2/4}\|z\|_2^m\,dz\le C t^{m/2}\int_{\|z\|_2\ge r_0/\sqrt{t}} e^{-\|z\|_2^2/8}\,dz\le C e^{-r_0^2/(16t)}$$
for all sufficiently small $t>0$. Combining the two estimates, we obtain $|I_2(t)|\le C e^{-c/t}$ for some constant $c>0$ and all sufficiently small $t>0$. This proves that $P_t f(0)=I_1(t)+I_2(t)=O(t^{a/2})$.
\end{proof}

\begin{theorem}\label{thm:existence-localization} As $t\to0^+$, 
$\|x_t\|_2=O(\sqrt t)$.

\end{theorem}

\begin{proof}
Since $P_t f(x_t)\le P_t f(0)$,~Lemma~\ref{lem:Ptf0-estimate} yields $P_t f(x_t)\le Ct^{a/2}$. On the other hand,~Lemma~\ref{lem:gaussian-lower-bound} gives $(P_t f)(x_t)\ge \|x_t\|_2^2+2nt$. Hence
$$\|x_t\|_2^2+2nt\le Ct^{a/2}.$$
This immediately shows that $x_t\to0$ as $t\to0^+$.

To obtain the sharper estimate, we first note that $\|x_t\|_2\to0$ as $t\to0^+$. Hence $x_t\in B_{r_0/2}(0)$ for all sufficiently small $t>0$. Using the local lower bound $f(y)\ge \frac12\Phi_a(y)$ for $y\in B_{r_0}(0)$, we obtain
\begin{equation}\label{eq:Ptf_ineq}
    P_t f(x_t)\ge \frac12\int_{B_{r_0}(0)} G_t(x_t-y)\,\Phi_a(y)\,dy-R_t,
\end{equation}
where
$$R_t:=\left|\int_{\mathbb R^n\setminus B_{r_0}(0)} G_t(x_t-y)\,f(y)\,dy\right|.$$
Since $\|x_t\|_2\le r_0/2$, we have $\|x_t-y\|_2\ge r_0/2$ for every $y\in\mathbb R^n\setminus B_{r_0}(0)$. Together with the polynomial growth bound $|f(y)|\le C(1+\|y\|_2^m)$ for some $m\ge2$, the same change-of-variables argument ($y=x_t+\sqrt{t}\,z$) as in the proof of~Lemma~\ref{lem:Ptf0-estimate} yields 
\begin{equation}\label{eq:Rt}
    R_t\le Ce^{-c/t}
\end{equation}
for some constants $C>0$.

Next, we write
$$P_t\Phi_a(x_t)=\int_{B_{r_0}(0)} G_t(x_t-y)\,\Phi_a(y)\,dy+\underbrace{\int_{\mathbb R^n\setminus B_{r_0}(0)} G_t(x_t-y)\,\Phi_a(y)\,dy}_{=:S_t}.$$
Since $\Phi_a$ has polynomial growth and $\|x_t-y\|_2\ge r_0/2$ on $\mathbb R^n\setminus B_{r_0}(0)$, the same Gaussian tail estimate~Eq.\eqref{eq:gaussian_tail} gives
\begin{equation}\label{eq:St}
    S_t\le Ce^{-c/t}.
\end{equation}
Combining~Eq.\eqref{eq:Ptf_ineq}-Eq.\eqref{eq:St}, we obtain
$$P_t f(x_t)\ge \frac12 P_t\Phi_a(x_t)-Ce^{-c/t}.$$

Now, since $y \mapsto G_t(x_t-y)$ is a probability density centered at $x_t$ and $\Phi_a$ is convex for $1\le a\le 2$, Jensen's inequality yields
$$P_t\Phi_a(x_t)\ge \Phi_a(x_t).$$
Hence $P_t f(x_t)\ge \frac12\Phi_a(x_t)-Ce^{-c/t}$. Combining this with $P_t f(x_t)\le P_t f(0)=O(t^{a/2})$, we obtain $\Phi_a(x_t)=O(t^{a/2})$. Since $\Phi_a(x_t)=\sum_{i=1}^n |(x_t)_i|^a\ge C\|x_t\|_2^a$ for some constant $C>0$, it follows that $\|x_t\|_2^a=O(t^{a/2})$, and hence $\|x_t\|_2=O(\sqrt t)$. This completes the proof.
\end{proof}

\subsection{Hessian lower bounds along the minimizing branch}
\begin{lemma}\label{lem:model-hessian-lower-bound}
The following hold:
\begin{itemize}
\item If $a=2$, then $\nabla_x^2(P_t\Phi_2)(x_t)=2I$ for all $t>0$.
\item If $1\le a<2$, then there exist constants $c_0>0$ and $t_0>0$ such that $\nabla_x^2(P_t\Phi_a)(x_t)\succeq c_0\,t^{(a-2)/2}I$ for all $0<t<t_0$.
\end{itemize}
\end{lemma}
Here, $\Phi_a(x):=\sum_{i=1}^n |x_i|^a$ denotes the leading profile from Definition~\ref{ass:leading-profile}, and $x_t$ is a family of minimizers satisfying $\|x_t\|_2=O(\sqrt t)$ by Theorem~\ref{thm:existence-localization}.

\begin{proof}
Since $\Phi_a$ is separable, the factorization of the heat kernel gives $P_t\Phi_a(x)=\sum_{i=1}^n p_t^a(x_i)$, where
$$p_t^a(s):=\frac{1}{\sqrt{4\pi t}}\int_{\mathbb R} e^{-(s-y)^2/(4t)}|y|^a\,dy.$$
Therefore, $\nabla_x^2(P_t\Phi_a)(x)=\operatorname{diag}((p_t^a)''(x_1),\dots,(p_t^a)''(x_n))$.

If $a=2$, then $p_t^2(s)=s^2+2t$ by Example~\ref{ex:quadratic}, so $(p_t^2)''(s)=2$ for all $s\in\mathbb R$. Hence $\nabla_x^2(P_t\Phi_2)(x_t)=2I$.

Now assume $1\le a<2$. Since $\|x_t\|_2=O(\sqrt t)$, there exist constants $M>0$ and $t_0>0$ such that $\|x_t\|_2\le M\sqrt t$ for all $0<t<t_0$. In particular, $|(x_t)_i|\le M\sqrt t$ for all $i=1,\dots,n$ and $0<t<t_0$. We distinguish two cases.

\emph{Case 1: $a=1$.} As in Example~\ref{ex:linear}, $(p_t^1)''(s)=\frac{1}{\sqrt{\pi t}}e^{-s^2/(4t)}$. Since $|(x_t)_i|\le M\sqrt t$, we have $e^{-((x_t)_i)^2/(4t)} \ge e^{-M^2/4}$. Hence, for every $i=1,\dots,n$ and $0<t<t_0$,
$$(p_t^1)''((x_t)_i)\ge \frac{1}{\sqrt{\pi t}}e^{-M^2/4}.$$
Therefore, $\nabla_x^2(P_t\Phi_1)(x_t)\succeq \frac{e^{-M^2/4}}{\sqrt\pi}\,t^{-1/2}I$.

\emph{Case 2: $1<a<2$.} By the scaling $y=\sqrt t\,z$, we have $p_t^a(s)=t^{a/2}\psi_a(s/\sqrt t)$, where
$$\psi_a(\xi):=\frac{1}{\sqrt{4\pi}}\int_{\mathbb R} e^{-(\xi-z)^2/4}|z|^a\,dz.$$
Differentiating twice gives $(p_t^a)''(s)=t^{(a-2)/2}\psi_a''(s/\sqrt t)$. By standard properties of convolution with the heat kernel, we can evaluate the second derivative using the distributional derivative of $|z|^a$. Since $1<a<2$, we have $(|z|^a)''=a(a-1)|z|^{a-2}$ in the distributional sense, with $|z|^{a-2}\in L^1_{\mathrm{loc}}(\mathbb R)$. Hence
$$\psi_a''(\xi)=\frac{a(a-1)}{\sqrt{4\pi}}\int_{\mathbb R} e^{-(\xi-z)^2/4}|z|^{a-2}\,dz>0 \qquad \text{for all }\xi\in\mathbb R.$$
Moreover, $\psi_a''$ is continuous, so
$$c_{a,M}:=\inf_{|\xi|\le M}\psi_a''(\xi)>0.$$
Since $|(x_t)_i|/\sqrt t\le M$, it follows that
$$(p_t^a)''((x_t)_i)\ge c_{a,M}\,t^{(a-2)/2}\qquad \text{for all }i=1,\dots,n,\ 0<t<t_0.$$
Therefore, $\nabla_x^2(P_t\Phi_a)(x_t)\succeq c_{a,M}\,t^{(a-2)/2}I$.
\end{proof}

\begin{lemma}\label{lem:remainder-hessian-small}
The remainder $r$ in the decomposition $f=\Phi_a+r$ satisfies
$$\|\nabla_x^2(P_t r)(x_t)\|=o\bigl(t^{(a-2)/2}\bigr)\qquad\text{as }t\to0^+.$$
\end{lemma}

\begin{proof}
Since $\|x_t\|_2=O(\sqrt t)$, there exist constants $M>0$ and $t_0>0$ such that $x_t=\sqrt t\,\xi_t$ and $\|\xi_t\|_2\le M$ for all $0<t<t_0$. Let $\rho\in(0,r_0)$ be fixed. For each pair $1\le i,j\le n$, write
$$\partial_{ij}(P_t r)(x_t)=\int_{\mathbb R^n}\partial_{ij}G_t(x_t-y)\,r(y)\,dy=:I_{ij}^{\mathrm{loc}}(t)+I_{ij}^{\mathrm{far}}(t),$$
where
$$I_{ij}^{\mathrm{loc}}(t):=\int_{B_\rho(0)}\partial_{ij}G_t(x_t-y)\,r(y)\,dy,\qquad I_{ij}^{\mathrm{far}}(t):=\int_{\mathbb R^n\setminus B_\rho(0)}\partial_{ij}G_t(x_t-y)\,r(y)\,dy.$$

We first estimate the local part $I_{ij}^{\mathrm{loc}}(t)$. The second derivative of the heat kernel satisfies
$$\partial_{ij}G_t(z)=t^{-1-n/2}K_{ij}\!\left(\frac{z}{\sqrt t}\right),$$
where
$$K_{ij}(z):=\frac{1}{(4\pi)^{n/2}}\left(\frac{z_iz_j}{4}-\frac{\delta_{ij}}{2}\right)e^{-\|z\|_2^2/4},$$
and $\delta_{ij}$ denotes the Kronecker delta. Since $|r(y)|\le \omega(\|y\|_2)\Phi_a(y)$ for $y\in B_\rho(0)$, the change of variables $y=\sqrt t\,z$ yields
$$I_{ij}^{\mathrm{loc}}(t)=t^{(a-2)/2}\int_{B_{\rho/\sqrt t}(0)}K_{ij}(\xi_t-z)\,\omega(\sqrt t\,\|z\|_2)\,\Phi_a(z)\,dz.$$
Since $\|\xi_t\|_2\le M$, the family $|K_{ij}(\xi_t-z)|\Phi_a(z)$ is dominated by the envelope 
$$\sup_{\|\xi\|_2\le M}|K_{ij}(\xi-z)|\Phi_a(z)$$
Moreover, $\omega(\sqrt t\,\|z\|_2)\to0$ as $t \to 0^+$ for each fixed $z\in\mathbb R^n$. Hence, by the dominated convergence theorem,
$$I_{ij}^{\mathrm{loc}}(t)=o\bigl(t^{(a-2)/2}\bigr)\qquad\text{as }t\to0^+.$$

Next we estimate the far-field part. Since $x_t\to0$, after possibly shrinking $t_0$ we may assume that $\|x_t\|_2\le \rho/2$ for all $0<t<t_0$. Thus, for $\|y\|_2\ge \rho$, one has $\|x_t-y\|_2\ge \rho/2$. Using the explicit formula for $\partial_{ij}G_t$, there exist constants $C,c>0$ such that
$$|\partial_{ij}G_t(x_t-y)|\le C\,t^{-1-n/2}\bigl(1+t^{-1}\|x_t-y\|_2^2\bigr)e^{-c\|x_t-y\|_2^2/t}$$
for all $\|y\|_2\ge \rho$ and $0<t<t_0$. Since $|r(y)|\le C(1+\|y\|_2^m)$ for some $m\ge2$, it follows that
$$|I_{ij}^{\mathrm{far}}(t)|\le C\int_{\mathbb R^n\setminus B_\rho(0)} t^{-1-n/2}\bigl(1+t^{-1}\|x_t-y\|_2^2\bigr)e^{-c\|x_t-y\|_2^2/t}(1+\|y\|_2^m)\,dy.$$
Now set $y=x_t+\sqrt t\,z$. Since $\|x_t\|_2\le \rho/2$, the condition $\|y\|_2\ge \rho$ implies $\|z\|_2\ge \rho/(2\sqrt t)$. Hence
$$|I_{ij}^{\mathrm{far}}(t)|\le C t^{-1}\int_{\|z\|_2\ge \rho/(2\sqrt t)} (1+\|z\|_2^2)e^{-c\|z\|_2^2}\bigl(1+\|x_t+\sqrt t\,z\|_2^m\bigr)\,dz.$$
Using $\|x_t+\sqrt t\,z\|_2\le C\sqrt t(1+\|z\|_2)$ and absorbing the polynomial factor into the Gaussian tail, we obtain
$$|I_{ij}^{\mathrm{far}}(t)|\le C t^{-1}e^{-c/t}.$$
In particular,
$$I_{ij}^{\mathrm{far}}(t)=o\bigl(t^{(a-2)/2}\bigr)\qquad\text{as }t\to0^+,$$
since $t^{-1}e^{-c/t}=o(t^\beta)$ for every $\beta\in\mathbb R$.

Combining the local and far-field estimates, we conclude that
$$\partial_{ij}(P_t r)(x_t)=o\bigl(t^{(a-2)/2}\bigr)\qquad\text{for all }1\le i,j\le n.$$
Hence $\|\nabla_x^2(P_t r)(x_t)\|=o\bigl(t^{(a-2)/2}\bigr)$, as claimed.
\end{proof}

\begin{theorem}\label{thm:main-nondegeneracy}
Let $x_t$ be any family of global minimizers of $P_t f$. By Theorem~\ref{thm:existence-localization}, such a family satisfies $\|x_t\|_2=O(\sqrt t)$ as $t\to0^+$. Then there exists $t_0>0$ such that, for every $0<t<t_0$, the Hessian $\nabla_x^2(P_t f)(x_t)$ is positive definite. More precisely, the following hold:
\begin{itemize}
\item If $a=2$, then there exists $c_0>0$ such that $\nabla_x^2(P_t f)(x_t)\succeq c_0 I$ for all $0<t<t_0$.
\item If $1\le a<2$, then there exists $c_0>0$ such that $\nabla_x^2(P_t f)(x_t)\succeq c_0\,t^{(a-2)/2}I$ for all $0<t<t_0$.
\end{itemize}
\end{theorem}

\begin{proof}
Using the decomposition $f=\Phi_a+r$ near the origin together with the global definition of the remainder term, we write
$$\nabla_x^2(P_t f)(x_t)=\nabla_x^2(P_t\Phi_a)(x_t)+\nabla_x^2(P_t r)(x_t).$$

If $a=2$, then Lemma~\ref{lem:model-hessian-lower-bound} gives $\nabla_x^2(P_t\Phi_2)(x_t)=2I$, while Lemma~\ref{lem:remainder-hessian-small} yields $\|\nabla_x^2(P_t r)(x_t)\|=o(1)$ as $t\to0^+$. Therefore, for all sufficiently small $t>0$,
$$\nabla_x^2(P_t f)(x_t)\succeq I.$$
This proves the claim in the quadratic case.

Now assume $1\le a<2$. By Lemma~\ref{lem:model-hessian-lower-bound}, there exist constants $c_1>0$ and $t_1>0$ such that
$$\nabla_x^2(P_t\Phi_a)(x_t)\succeq c_1\,t^{(a-2)/2}I,\qquad \forall 0<t<t_1.$$
On the other hand, Lemma~\ref{lem:remainder-hessian-small} implies that
$$\|\nabla_x^2(P_t r)(x_t)\|=o\bigl(t^{(a-2)/2}\bigr)\qquad\text{as }t\to0^+.$$
Hence there exists $t_2>0$ such that
$$\|\nabla_x^2(P_t r)(x_t)\|\le \frac{c_1}{2}\,t^{(a-2)/2},\qquad \forall 0<t<t_2.$$
Therefore, for all $0<t<\min\{t_1,t_2\}$,
$$\nabla_x^2(P_t f)(x_t)\succeq \frac{c_1}{2}\,t^{(a-2)/2}I.$$
This proves the claim in the subquadratic case.
\end{proof}

\subsection{Local solvability of the continuation equation}
Recall Definitions~\ref{def:critical-branch} and \ref{def:nondegenerate-branch}. Fix a family $\{x_t\}_{0<t<t_0}$ of global minimizers from Theorem~\ref{thm:main-nondegeneracy}. Since $\nabla_x^2(P_t f)(x_t)\succ 0$ for all sufficiently small $t>0$, the continuation equation is locally well posed near each point $(t_*,x_{t_*})$ on the selected minimizing family.

\begin{proposition}\label{prop:local-branch}
There exists $t_0>0$ such that the following holds. For every $t_*\in(0,t_0)$, there exist $\varepsilon>0$ and a unique $C^1$ map $y:(t_*-\varepsilon,t_*+\varepsilon)\to\mathbb R^n$ such that
$$
y(t_*)=x_{t_*}
\qquad\text{and}\qquad
\nabla_x(P_t f)(y(t))=0
\qquad\text{for all }t\in(t_*-\varepsilon,t_*+\varepsilon).
$$
After possibly shrinking $\varepsilon$, one also has
$$
\nabla_x^2(P_t f)(y(t))\succ 0
\qquad\text{for all }t\in(t_*-\varepsilon,t_*+\varepsilon).
$$
In particular, $y$ is a nondegenerate local minimizing branch on $(t_*-\varepsilon,t_*+\varepsilon)$ in the sense of Definition~\ref{def:nondegenerate-branch}.
\end{proposition}

\begin{proof}
Set $F(t,x):=\nabla_x(P_t f)(x)$. By Theorem~\ref{thm:main-nondegeneracy}, there exists $t_0>0$ such that
$$
\nabla_x^2(P_{t_*}f)(x_{t_*})\succ 0
\qquad\text{for every }t_*\in(0,t_0).
$$
Hence
$$
F(t_*,x_{t_*})=0
\qquad\text{and}\qquad
D_xF(t_*,x_{t_*})=\nabla_x^2(P_{t_*}f)(x_{t_*})
$$
is invertible. The implicit function theorem therefore yields $\varepsilon>0$ and a unique $C^1$ map $y$ with $y(t_*)=x_{t_*}$ and
$$
F(t,y(t))=0
\qquad\text{for all }t\in(t_*-\varepsilon,t_*+\varepsilon).
$$
That is,
$$
\nabla_x(P_t f)(y(t))=0
\qquad\text{for all }t\in(t_*-\varepsilon,t_*+\varepsilon).
$$
The positivity of $\nabla_x^2(P_t f)(y(t))$ then follows from continuity after shrinking $\varepsilon$ if necessary.
\end{proof}

\begin{corollary}\label{cor:continuation-equation}
Let $y(t)$ be the local branch from Proposition~\ref{prop:local-branch}. Then
\begin{equation}\label{eq:continuation-local}
\partial_t\nabla_x(P_t f)(y(t))+\nabla_x^2(P_t f)(y(t))\,\dot y(t)=0
\qquad\text{for all }t\in(t_*-\varepsilon,t_*+\varepsilon).
\end{equation}
Equivalently,
\begin{equation}\label{eq:continuation-local-inverse}
\dot y(t)=-\bigl(\nabla_x^2(P_t f)(y(t))\bigr)^{-1}\partial_t\nabla_x(P_t f)(y(t))
\qquad\text{for all }t\in(t_*-\varepsilon,t_*+\varepsilon).
\end{equation}
\end{corollary}

\begin{proof}
Differentiate the identity $\nabla_x(P_t f)(y(t))=0$ with respect to $t$. This gives Equation~\eqref{eq:continuation-local}. Since $\nabla_x^2(P_t f)(y(t))\succ 0$ by Proposition~\ref{prop:local-branch}, the Hessian is invertible, and Equation~\eqref{eq:continuation-local-inverse} follows.
\end{proof}

\subsection{Branch energy and continuation obstructions}

Recall Definitions~\ref{def:critical-branch}, \ref{def:nondegenerate-branch}, and \ref{def:maximal-branch}. We begin with an energy identity along a critical-point branch.

\begin{lemma}\label{lem:branch-energy}
Let $I\subset(0,\infty)$ be an interval, and let $x:I\to\mathbb R^n$ be a critical-point branch on $I$ in the sense of Definition~\ref{def:critical-branch}. Define
$$
E(t):=(P_t f)(x_t)
$$
for $t\in I$. Then $E\in C^1(I)$ and
$$
E'(t)=\partial_t(P_t f)(x_t)=\Delta(P_t f)(x_t)
$$
for all $t\in I$.
\end{lemma}

\begin{proof}
Since $P_t f$ is smooth in $(t,x)$ for every $t>0$, the map $E(t)=(P_t f)(x_t)$ is $C^1$. By the chain rule,
$$
E'(t)=\partial_t(P_t f)(x_t)+\nabla_x(P_t f)(x_t)\cdot \dot x_t.
$$
Because $x_t$ is a critical-point branch, the second term vanishes, so
$$
E'(t)=\partial_t(P_t f)(x_t).
$$
Finally, $u(t,x):=(P_t f)(x)$ solves the heat equation $\partial_t u=\Delta_x u$, hence
$$
E'(t)=\Delta(P_t f)(x_t).
$$
\end{proof}

\begin{corollary}\label{cor:energy-crossing}
Let $x^{(i)},x^{(j)}:I\to\mathbb R^n$ be two critical-point branches on $I$, and define
$$
E_i(t):=(P_t f)(x_t^{(i)}),\qquad
E_j(t):=(P_t f)(x_t^{(j)}),\qquad
D_{ij}(t):=E_i(t)-E_j(t).
$$
Then
$$
D_{ij}'(t)=\Delta(P_t f)(x_t^{(i)})-\Delta(P_t f)(x_t^{(j)}).
$$
In particular, if for some $t_*\in I$,
$$
D_{ij}(t_*)=0
\qquad\text{and}\qquad
D_{ij}'(t_*)\neq 0,
$$
then $t_*$ is an isolated crossing time, and there exists $\varepsilon>0$ such that $D_{ij}(t)$ changes sign on $(t_*-\varepsilon,t_*+\varepsilon)$.
\end{corollary}

\begin{proof}
Apply Lemma~\ref{lem:branch-energy} to each branch and subtract the two identities. The final assertion follows from the one-dimensional transversality condition.
\end{proof}

We now turn to continuation beyond the small-$t$ regime.

\begin{theorem}\label{thm:continuation-alternative}
Let $x:(0,T_{\max})\to\mathbb R^n$ be a nondegenerate local minimizing branch with $0<T_{\max}<\infty$, and assume that
$$
x_t\to \bar x
\qquad\text{as }t\to T_{\max}^-
$$
for some $\bar x\in\mathbb R^n$. Then the following are equivalent:
\begin{enumerate}
\item[(i)] The branch $x_t$ extends across $T_{\max}$ as a nondegenerate local minimizing branch.
\item[(ii)] The Hessian remains uniformly nondegenerate near $T_{\max}$, namely,
$$
\liminf_{t\to T_{\max}^-}\lambda_{\min}\bigl(\nabla_x^2(P_t f)(x_t)\bigr)>0.
$$
\end{enumerate}
In particular, if the branch is maximal, then
$$
\liminf_{t\to T_{\max}^-}\lambda_{\min}\bigl(\nabla_x^2(P_t f)(x_t)\bigr)=0.
$$
\end{theorem}

\begin{proof}
Assume first that (ii) holds. Then there exist $c_0>0$ and $t_1<T_{\max}$ such that
$$
\lambda_{\min}\bigl(\nabla_x^2(P_t f)(x_t)\bigr)\ge c_0
\qquad\text{for all }t\in(t_1,T_{\max}).
$$
Set
$$
F(t,x):=\nabla_x(P_t f)(x).
$$
Since $P_t f$ is smooth in $(t,x)$ for $t>0$, the map $F$ is continuous in $t$ and $C^1$ in $x$ near $(T_{\max},\bar x)$. Because $x_t\to\bar x$ and $F(t,x_t)=0$ for all $t<T_{\max}$, continuity gives
$$
F(T_{\max},\bar x)=0.
$$
Moreover, by continuity of the Hessian and the lower bound along the branch,
$$
\nabla_xF(T_{\max},\bar x)=\nabla_x^2(P_{T_{\max}}f)(\bar x)\succ 0.
$$
Hence $\nabla_xF(T_{\max},\bar x)$ is invertible. By the implicit function theorem, there exist $\varepsilon>0$, a neighborhood $U$ of $\bar x$, and a unique $C^1$ map
$$
y:(T_{\max}-\varepsilon,T_{\max}+\varepsilon)\to U
$$
such that
$$
F(t,y(t))=0
\qquad\text{and}\qquad
y(T_{\max})=\bar x.
$$
Since $x_t\to\bar x$, we have $x_t\in U$ for all $t$ sufficiently close to $T_{\max}$, and by uniqueness in the implicit function theorem,
$$
x_t=y(t)
\qquad\text{for all }t\in(T_{\max}-\varepsilon,T_{\max}).
$$
Shrinking $\varepsilon$ if necessary, continuity of the Hessian implies
$$
\nabla_x^2(P_t f)(y(t))\succ 0
\qquad\text{for all }t\in(T_{\max}-\varepsilon,T_{\max}+\varepsilon).
$$
Thus $y$ extends the original branch across $T_{\max}$ as a nondegenerate local minimizing branch. Therefore (ii)$\Rightarrow$(i).

Conversely, assume that (i) holds. Then there exist $\varepsilon>0$ and a nondegenerate local minimizing branch
$$
\widetilde x:(T_{\max}-\varepsilon,T_{\max}+\varepsilon)\to\mathbb R^n
$$
such that
$$
\widetilde x_t=x_t
\qquad\text{for all }t\in(T_{\max}-\varepsilon,T_{\max}).
$$
Since
$$
\nabla_x^2(P_t f)(\widetilde x_t)\succ 0
\qquad\text{for all }t\in(T_{\max}-\varepsilon,T_{\max}+\varepsilon),
$$
continuity of the Hessian implies that there exist $c_0>0$ and $\delta>0$ such that
$$
\lambda_{\min}\bigl(\nabla_x^2(P_t f)(\widetilde x_t)\bigr)\ge c_0
\qquad\text{for all }t\in(T_{\max}-\delta,T_{\max}+\delta).
$$
Hence
$$
\liminf_{t\to T_{\max}^-}\lambda_{\min}\bigl(\nabla_x^2(P_t f)(x_t)\bigr)\ge c_0>0.
$$
Thus (i)$\Rightarrow$(ii).

Finally, if the branch is maximal, then (i) fails. Therefore (ii) also fails. Since
$$
\lambda_{\min}\bigl(\nabla_x^2(P_t f)(x_t)\bigr)>0
\qquad\text{for all }t\in(0,T_{\max}),
$$
we necessarily have
$$
\liminf_{t\to T_{\max}^-}\lambda_{\min}\bigl(\nabla_x^2(P_t f)(x_t)\bigr)=0.
$$
\end{proof}

\section{Discussion}
In this work, we studied nonsmooth and globally nonconvex minima through the family of smooth objectives generated by heat-kernel regularization. Rather than analyzing the original singular landscape directly, we examined how minimizers of the regularized objectives behave as the smoothing scale $t$ tends to zero, and used this branch-wise viewpoint to recover the local geometry of the original minimum.

A central message of our analysis is that heat-kernel regularization is not merely a smoothing device for numerical convenience. It also reveals the local leading-order structure of a nonsmooth minimum through the asymptotic behavior of the regularized Hessian. In the quadratic regime ($a=2$), the Hessian remains at a finite positive scale along the minimizing branch, whereas in the subquadratic cusp regime ($1\le a<2$), it blows up at a controlled rate. In both cases, the regularized Hessian remains positive definite for sufficiently small $t>0$, so the continuation equation is not merely formal but locally well posed near the original minimizer. More precisely, every sufficiently small-scale minimizer lies on a unique local $C^1$ critical branch, and along any such branch the associated branch energy satisfies the identity
$\frac{d}{dt}(P_t f)(x_t)=\partial_t(P_t f)(x_t)=\Delta(P_t f)(x_t)$. Thus heat-kernel regularization restores not only curvature, but also a local continuation structure and an intrinsic energy law along the branch.

From this perspective, the main contribution of the present work is not algorithmic in the usual sense of nonsmooth optimization, such as complexity estimates or step-size design, but analytical. Rather than working directly with generalized gradients or stationarity notions for the original objective, our approach passes to a smooth one-parameter family and studies its minimizing branches. This yields a rigorous second-order framework for continuation-based analysis, identifies branch energy as a natural quantity for comparing competing local branches, and shows that under a finite-limit hypothesis a nondegenerate local minimizing branch can fail to extend past a finite terminal scale only through loss of uniform Hessian nondegeneracy. In this sense, heat-kernel regularization serves as a theoretical bridge from singular nonsmooth minima to traceable smooth branches, while also clarifying the structural mechanisms that govern local continuation and its possible breakdown.

Several directions remain open. A first issue is global path construction. Our results identify and control the branch in the small-\(t\) regime, but they do not guarantee the existence of a single globally smooth minimizing path over the full range of smoothing scales. In nonconvex settings, competing branches may coexist and exchange global optimality, as illustrated in Example~\ref{ex:discontinuous_branch}. It would therefore be interesting to develop a global branch-selection theory, possibly allowing piecewise-smooth continuation paths with switching times, and to characterize structural conditions under which a canonical global path exists.

A second direction is the extension to more general local profiles. In this paper, the leading-order behavior is modeled by the isotropic profile $\Phi_a(x)=\sum_{i=1}^n |x_i|^a$ with $a\in[1,2]$. Allowing anisotropic exponents, mixed terms, or more general homogeneous profiles would lead to direction-dependent curvature scales and may require a finer spectral analysis of the regularized Hessian. Moreover, the highly singular regime $a<1$ lies outside the scope of the present argument, since the loss of convexity and the stronger singularity invalidate the current estimates. Understanding that regime may require more explicit heat-kernel representations and different tools for handling the resulting singular structure.

\section{Conclusions}\label{sec:conclusions}
We analyzed heat-kernel regularization near nonsmooth minimizers and showed that the regularized minimizers localize at the natural heat scale $O(\sqrt t)$. We further proved that the regularized Hessian remains positive definite along the minimizing branch for all sufficiently small $t>0$, with asymptotic behavior determined by the local leading-order exponent $a$: a finite positive limit in the quadratic case $a=2$, and a controlled blow-up in the subquadratic case $1\le a<2$. In particular, every sufficiently small-scale minimizer lies on a unique local $C^1$ critical branch, so that the continuation equation is rigorously valid near the original minimizer rather than merely formal. We also identified an intrinsic branch-energy identity and showed, under a finite-limit hypothesis, that continuation past a terminal scale is obstructed only by loss of uniform Hessian nondegeneracy. These results provide a rigorous continuation-based second-order framework for studying nonsmooth minima through smooth regularized objectives.
\bibliographystyle{unsrt}  
\bibliography{references}

@article{garmanjani2013smoothing,
  title={Smoothing and worst-case complexity for direct-search methods in nonsmooth optimization},
  author={Garmanjani, Roholla and Vicente, Lu{\'\i}s Nunes},
  journal={IMA Journal of Numerical Analysis},
  volume={33},
  number={3},
  pages={1008--1028},
  year={2013},
  publisher={OUP}
}

@article{wang2022differentially,
  title={Differentially private SGD with non-smooth losses},
  author={Wang, Puyu and Lei, Yunwen and Ying, Yiming and Zhang, Hai},
  journal={Applied and Computational Harmonic Analysis},
  volume={56},
  pages={306--336},
  year={2022},
  publisher={Elsevier}
}

@article{nesterov2017random,
  title={Random gradient-free minimization of convex functions},
  author={Nesterov, Yurii and Spokoiny, Vladimir},
  journal={Foundations of Computational Mathematics},
  volume={17},
  number={2},
  pages={527--566},
  year={2017},
  publisher={Springer}
}

@article{iwakiri2022single,
  title={Single loop gaussian homotopy method for non-convex optimization},
  author={Iwakiri, Hidenori and Wang, Yuhang and Ito, Shinji and Takeda, Akiko},
  journal={Advances in Neural Information Processing Systems},
  volume={35},
  pages={7065--7076},
  year={2022}
}

@inproceedings{xu2025global,
  title={Global Optimization with a Power-Transformed Objective and Gaussian Smoothing},
  author={Xu, Chen},
  booktitle={International Conference on Machine Learning},
  pages={69189--69216},
  year={2025},
  organization={PMLR}
}

@article{allgower1993continuation,
  title={Continuation and path following},
  author={Allgower, Eugene L and Georg, Kurt},
  journal={Acta numerica},
  volume={2},
  pages={1--64},
  year={1993},
  publisher={Cambridge University Press}
}

@article{hao2022adaptive,
  title={An adaptive homotopy tracking algorithm for solving nonlinear parametric systems with applications in nonlinear ODEs},
  author={Hao, Wenrui},
  journal={Applied Mathematics Letters},
  volume={125},
  pages={107767},
  year={2022},
  publisher={Elsevier}
}

@article{seguin2022continuation,
  title={Continuation methods for Riemannian optimization},
  author={S{\'e}guin, Axel and Kressner, Daniel},
  journal={SIAM Journal on Optimization},
  volume={32},
  number={2},
  pages={1069--1093},
  year={2022},
  publisher={SIAM}
}

@article{starnes2023gaussian,
  title={Gaussian smoothing gradient descent for minimizing functions (gsmoothgd)},
  author={Starnes, Andrew and Dereventsov, Anton and Webster, Clayton},
  journal={arXiv preprint arXiv:2311.00521},
  year={2023}
}

@inproceedings{mobahi2015theoretical,
  title={A theoretical analysis of optimization by gaussian continuation},
  author={Mobahi, Hossein and Fisher III, John},
  booktitle={Proceedings of the AAAI Conference on Artificial Intelligence},
  volume={29},
  number={1},
  year={2015}
}

@article{gu2026deep,
  title={Deep Predictor-Corrector Networks for Robust Parameter Estimation in Non-autonomous System with Discontinuous Inputs},
  author={Gu, Gyeongwan and Hyun, Jinwoo and Jo, Hyeontae and Kim, Jae Kyoung},
  journal={arXiv preprint arXiv:2603.12965},
  year={2026}
}

@article{ilersich2025deep,
  title={Deep learning with Gaussian continuation},
  author={Ilersich, Andrew F and Nair, Prasanth B},
  journal={Foundations of Data Science},
  volume={7},
  number={3},
  pages={790--813},
  year={2025},
  publisher={Foundations of Data Science}
}

@book{stein2009real,
  title={Real analysis: measure theory, integration, and Hilbert spaces},
  author={Stein, Elias M and Shakarchi, Rami},
  year={2009},
  publisher={Princeton University Press}
}

@book{brezis2011functional,
  title={Functional analysis, Sobolev spaces and partial differential equations},
  author={Brezis, Haim and Br{\'e}zis, Haim},
  volume={2},
  number={3},
  year={2011},
  publisher={Springer}
}

@article{fuduli2004minimizing,
  title={Minimizing nonconvex nonsmooth functions via cutting planes and proximity control},
  author={Fuduli, Antonio and Gaudioso, Manlio and Giallombardo, Giovanni},
  journal={SIAM journal on optimization},
  volume={14},
  number={3},
  pages={743--756},
  year={2004},
  publisher={SIAM}
}

@article{nesterov2005smooth,
  title={Smooth minimization of non-smooth functions},
  author={Nesterov, Yu},
  journal={Mathematical programming},
  volume={103},
  number={1},
  pages={127--152},
  year={2005},
  publisher={Springer}
}

@article{kiwiel2007convergence,
  title={Convergence of the gradient sampling algorithm for nonsmooth nonconvex optimization},
  author={Kiwiel, Krzysztof C},
  journal={SIAM Journal on Optimization},
  volume={18},
  number={2},
  pages={379--388},
  year={2007},
  publisher={SIAM}
}

@article{ma2022beyond,
  title={Beyond the quadratic approximation: The multiscale structure of neural network loss landscapes},
  author={Ma, Chao and Kunin, Daniel and Wu, Lei and Ying, Lexing},
  journal={arXiv preprint arXiv:2204.11326},
  year={2022}
}

@article{jia2025first,
  title={First-order methods for nonsmooth nonconvex functional constrained optimization with or without slater points},
  author={Jia, Zhichao and Grimmer, Benjamin},
  journal={SIAM Journal on Optimization},
  volume={35},
  number={2},
  pages={1300--1329},
  year={2025},
  publisher={SIAM}
}

@article{jo2026neural,
  title={Neural Network--Based Parameter Estimation for Nonautonomous Differential Equations with Discontinuous Signals},
  author={Jo, Hyeontae and Josi{\'c}, Kre{\v{s}}imir and Kim, Jae Kyoung},
  journal={SIAM Journal on Applied Mathematics},
  volume={86},
  number={1},
  pages={327--347},
  year={2026},
  publisher={SIAM}
}

@article{ishige2026preservation,
  title={Preservation of F-convexity under the heat flow},
  author={Ishige, Kazuhiro and Petitt, Troy and Salani, Paolo},
  journal={arXiv preprint arXiv:2603.10920},
  year={2026}
}

@article{burke1993weak,
  title={Weak sharp minima in mathematical programming},
  author={Burke, James V and Ferris, Michael C},
  journal={SIAM Journal on Control and Optimization},
  volume={31},
  number={5},
  pages={1340--1359},
  year={1993},
  publisher={SIAM}
}

@book{bonnans2013perturbation,
  title={Perturbation analysis of optimization problems},
  author={Bonnans, J Frederic and Shapiro, Alexander},
  year={2013},
  publisher={Springer Science \& Business Media}
}

@article{drusvyatskiy2013tilt,
  title={Tilt stability, uniform quadratic growth, and strong metric regularity of the subdifferential},
  author={Drusvyatskiy, Dmitriy and Lewis, Adrian S},
  journal={SIAM Journal on Optimization},
  volume={23},
  number={1},
  pages={256--267},
  year={2013},
  publisher={SIAM}
}

@article{chieu2021quadratic,
  title={Quadratic growth and strong metric subregularity of the subdifferential via subgradient graphical derivative},
  author={Chieu, Nguyen Huy and Hien, Le Van and Nghia, Tran TA and Tuan, Ha Anh},
  journal={SIAM Journal on Optimization},
  volume={31},
  number={1},
  pages={545--568},
  year={2021},
  publisher={SIAM}
}

@article{davis2025local,
  title={A local nearly linearly convergent first-order method for nonsmooth functions with quadratic growth},
  author={Davis, Damek and Jiang, Liwei},
  journal={Foundations of Computational Mathematics},
  volume={25},
  number={3},
  pages={943--1024},
  year={2025},
  publisher={Springer}
}

@article{khanh2025local,
  title={Local Minimizers of Nonconvex Functions in Banach Spaces via Moreau Envelopes: PD Khanh et al.},
  author={Khanh, Pham Duy and Khoa, Vu VH and Mordukhovich, Boris S and Phat, Vo Thanh},
  journal={Vietnam Journal of Mathematics},
  volume={53},
  number={4},
  pages={803--813},
  year={2025},
  publisher={Springer}
}

@book{cannon1984one,
  title={The one-dimensional heat equation},
  author={Cannon, John Rozier},
  number={23},
  year={1984},
  publisher={Cambridge University Press}
}

@book{stein1970singular,
  title={Singular integrals and differentiability properties of functions},
  author={Stein, Elias M},
  number={30},
  year={1970},
  publisher={Princeton university press}
}

@inproceedings{seo2010heat,
  title={Heat kernel smoothing using Laplace-Beltrami eigenfunctions},
  author={Seo, Seongho and Chung, Moo K and Vorperian, Houri K},
  booktitle={International Conference on Medical Image Computing and Computer-Assisted Intervention},
  pages={505--512},
  year={2010},
  organization={Springer}
}

@book{lindeberg2013scale,
  title={Scale-space theory in computer vision},
  author={Lindeberg, Tony},
  volume={256},
  year={2013},
  publisher={Springer Science \& Business Media}
}

@article{yang2014mollification,
  title={A mollification regularization method for the inverse spatial-dependent heat source problem},
  author={Yang, Fan and Fu, Chu-Li},
  journal={Journal of computational and applied mathematics},
  volume={255},
  pages={555--567},
  year={2014},
  publisher={Elsevier}
}

@article{chung2019heat,
  title={Heat kernel smoothing in irregular domains},
  author={Chung, M and Wang, Yanli},
  journal={Institute for Mathematical Sciences, National University of Singapore},
  pages={181--210},
  year={2019},
  publisher={World Scientific}
}

@article{luo2024regularization,
  title={The regularization continuation method for optimization problems with nonlinear equality constraints},
  author={Luo, Xin-long and Xiao, Hang and Zhang, Sen},
  journal={Journal of Scientific Computing},
  volume={99},
  number={1},
  pages={17},
  year={2024},
  publisher={Springer}
}

\end{document}